\documentclass[11pt]{amsart}
\usepackage{amsthm,amsmath,amssymb,enumerate}

\RequirePackage[colorlinks,citecolor=blue,urlcolor=blue]{hyperref}
\allowdisplaybreaks[3]

\numberwithin{equation}{section}
\theoremstyle{plain}
\theoremstyle{plain}

\newtheorem{theorem}{Theorem}[section]
\newtheorem{lemma}{Lemma}[section]
\newtheorem{proposition}{Proposition}[section]
\theoremstyle{remark}
\newtheorem{remark}{Remark}[section]
\newcommand{\rd}{\,\mathrm{d}} 
\newcommand{\ri}{\mathrm{i}} 

\DeclareMathOperator\var{var}
\newcommand{\rP}{\mathbb{P}} 
\newcommand{\rE}{\mathbb{E}}

\newcommand{\cvgdist}{\xrightarrow{\text{\upshape\tiny D}}} 
\renewcommand{\Re}{{\rm Re}}
\newcommand{\re}{\mathrm{e}}
\newcommand{\cis}[2]{#1 \re^{\ri #2}}
\newcommand{\braket}[2]{{\langle{#1|#2}\rangle}}

\begin{document}\title[An exact asymptotic]{An asymptotic variance of the self-intersections of random walks}

\author[G. Deligiannidis]{George Deligiannidis}
\address{Department of Mathematics\\University of Leicester\\LE1 7RH, UK}
\email{gd84@le.ac.uk}

\author[S. Utev]{Sergey Utev}
\address{School of Mathematical Sciences\\University of Nottingham\\NG7 2RD, UK}
\email{sergey.utev@nottingham.ac.uk}

\begin{abstract}
We present a Darboux-Wiener type lemma and apply it to obtain an exact asymptotic for the variance of the self-intersection of one and two-dimensional random walks. As a corollary, we obtain a central limit theorem for random walk in random scenery conjectured by Kesten and Spitzer~\cite{Kest79}.
\end{abstract}
\maketitle

\section{Introduction}
Consider a random walk \ $S_0 = 0,$ \ $S_n =X_1+\ldots+X_n$
with iid increments $\{X_i$, $i\in \mathbb{N}\}$ of $\mathbb{Z}^d$-valued random variables, for $d=1,2$.
Let $V_n$ be the number of self-intersections of the random walk up to time $n$,
\begin{equation}\label{vn}
V_n = \sum\limits_{i,j=0}^n \mathbf{1}_{S_i = S_j}.
\end{equation}
The asymptotic moments of $V_n$ have been treated extensively for their close relation with the Edwards model and self-avoiding walks (see Lawler~\cite{Law91}), and their importance in the limit theory of {\sl random walk in random scenery} (see \cite{Kest79}).

The conjecture that the variance of the self-intesections is of the order of $O(n^2)$ has been around for more than thirty years and its origins can be traced back to early work by Varadhan~\cite{Var69app} and Symanzik~\cite{Symanzik69}. Particular cases have been solved in the literature, such as the two-dimensional simple random walk which was proved in \cite[Prop.~6.4.1]{Law91}. For the case of two-dimensional recurrent random walk Bolthausen~\cite{Bolt89} developed a methodology based on the asymptotic analysis of the generating function $\sum_{i=1}^n \lambda^n \var(V_n)$, $\lambda \in [0,1)$ and the Tauberian theorem. His method allows to treat symmetrized random walk, while for the general case only the weaker bound $O(n^2\log n)$ can be obtained (further explanations given at the beginning of Section~3). A similar approach appeared more recently in \cite{Cerny07}, where once again the $n^2$ bound was only proven conclusively for special cases.
In particular the conjecture remained open until now.

In this paper we shall present a different approach, based on a Darboux-Wiener type result,  Lemma~\ref{tauberian}, which serves as a powerful alternative to the Tauberian theorem. While we too consider the asymptotics of the generating function, we allow the parameter $\lambda$ to be {\sl complex}, and using Cauchy's formula we are able to completely avoid the monotonicity restriction imposed by the Tauberian theorem. In this way we prove the $n^2$ bound without additional symmetry conditions, and we actually obtain the exact asymptotics rather than just a bound.
We also prove the conjecture in one dimension for random walks attracted to the symmetric Cauchy law. As a corollary we prove a functional central limit theorem for random walk in random scenery conjectured by Kesten and Spitzer~\cite{Kest79}.

Finally, we also present a direct Fourier approach as an alternative method of proof for the variance conjecture.

The rest of the paper is structured as follows. In Section~2 we present our main results. Then the proofs are given in Sections~3 and 4.

\section{Main results}
Let $f(t)$, for $t\in J=[-\pi,\pi)^d$, be the characteristic function of the $X_i$. We assume the random walk is {\sl strongly aperiodic} in the sense that
there is no proper subgroup $L$ of $\mathbb{Z}^d$ such that for some $x \in \mathbb{Z}^d$ with
$\rP(X_i = x) >0$ and $\rP(X_i-x \in L) = 1$. This assumption then implies that for $t\in J$,
\ $f(t)=1$ if and only if  $t=0$.

Theorems~\ref{main} and \ref{thmbounds} concern the asymptotic variance of $V_n$.
\begin{theorem}\label{main}
(i) Let $d=1$ and $f(t) = 1 - \gamma |t| + R(t),$ where $R(t) = o(|t|)$, as $t\to 0$. Then,
$$\var(V_n) \thicksim 4\left( \frac{1}{12\gamma^2} + \frac{1}{\pi^{2}\gamma^{2}} \right)n^2.$$
(ii) Let $d=2$ and assume that the $X_i$ have a non-singular covariance matrix $\Sigma$, such that \
$ f(t) = 1 - \frac{1}{2}\braket{\Sigma t}{t} + R(t),$ where $R(t) = o(|t|^2)$, as $|t| \to 0$.
Then,
$$\var(V_n) \thicksim 4(2\pi)^{-2}|\Sigma|^{-1}(1+\kappa)n^2, \mbox{ where }$$
$$\kappa \equiv \int_0^\infty\!\!\!\int_0^\infty
\frac{\rd r \rd s}{(1+r)(1+s)\sqrt{(1+r+s)^2 - 4rs}} - \frac{\pi^2}{6}.$$
\end{theorem}
A direct Fourier approach, similar to the local limit theorem technique in \cite{Law91}, is enough to prove the variance conjecture without resorting to Tauberian type results.
\begin{theorem}\label{thmbounds}
Let $X_j$ be independent $\mathbb{Z}^d$ variables with characteristic function $f_j(t)$, and $S_n$, $V_n$ be defined as above.
Assume in addition that either
(i) $d=1$, $|f_j(t)|\leq \re^{-\alpha |t|}$ and $|1-f_j(t)|\leq \alpha |t|$ for some positive constant $\alpha$, or
(ii) $d=2$, $|f_j(t)|\leq \re^{-\alpha |t|^2}$ and $|1-f_j(t)|\leq \alpha |t|^2$ for some positive constant $\alpha$.
Then
$$\var(V_n) \leq C(\alpha)n^2.$$
\end{theorem}
Suppose further, that $\xi(\alpha)$, indexed by $\alpha \in \mathbb{Z}^d$, are iid, real random variables, independent of the $X_i$ with $\rE \xi(\alpha) = 0$, $\rE\xi(\alpha)^2 = \sigma^2 >0$.
Then by {\slshape random walk in random scenery}
we shall mean the process
$$Z_0=0,\quad Z_n = \sum_{k=1}^n \xi(S_k), \quad n\geq 1.$$

Various limit theorems have appeared in the literature concerning weak limits of the process
$Y_n(t) = Z_{[nt]}/c_n$, $t\in [0,1]$ where $c_n$ is some normalizing sequence.
For random walks satisfying the assumptions of Theorem~\ref{main}(ii), \cite{Bolt89} showed that
$\{Y_n(\cdot)\}_n$ converges weakly in $D[0,\infty)$ to Brownian motion under the normalizing sequence $c_n = \sqrt{n\log n}$.
For $d=1$, with $X_i$, $\xi(\alpha)$ in the domain of attraction of distinct stable laws with parameters $\alpha \in (1, 2]$ and $\beta\in (0,2]$ respectively, \cite{Kest79} obtained a non-Gaussian limiting process.
The case $\alpha < 1$, transient random walk, is much simpler and was treated earlier in \cite{Spitzer76}.
The remaining case, $\alpha=1$, was conjectured by Kesten and Spitzer~\cite{Kest79} to converge to Brownian motion under the normalization $c_n = \sqrt{n \log n}$.
The proof of this is given in the following theorem.
\begin{theorem}\label{clt}
Suppose the random walk in Theorem~\ref{main}(i) and an independent i.i.d. sequence
$\{\xi(\alpha)\}_{\alpha \in \mathbb{Z}}$, with $\rE\xi(\alpha) =0$ and $\rE \xi(\alpha)^2 = \sigma^2>0$,
are now defined on the product probability space $\Sigma \times \Xi$.
Let
$$Y_n(t) = Y_n(t,\omega) = \sqrt{\pi \gamma} \sum_{i=0}^{[nt]}\xi(S_i(\omega))/\sigma \sqrt{2n\log n},\quad t\in [0,1].$$
Then the laws of $(Y_n(\cdot, \omega)_{n\geq 0}$ converge weakly in $D[0,1]$ to the Wiener measure for a.e. random walk path $\omega\in \Sigma$.
\end{theorem}
\begin{remark}
Almost sure convergence was partly motivated by \cite{Plant08}.
A similar version of the result in \cite{Bolt89} also holds.
\end{remark}

\section{Proofs of Theorems~\ref{main},\ref{thmbounds}\label{proofs}}
The variance of $V_n$ is given by
\begin{equation*}
 \var(V_n) = 4 \sum\limits_{H} \Big[
 \rP(S_{i_1} = S_{j_1},S_{i_2}=S_{j_2}) - \rP(S_{i_1}=S_{j_1})\rP(S_{i_2}=S_{j_2}) \Big],
\end{equation*}
where $H$ is the the set of 4-tuples
$$H = \{(i_1,j_1,i_2,j_2):0\leq i_1,j_1,i_2,j_2\leq n, i_1<j_1,i_2 <j_2\}, $$
which we partition into six sets
\begin{equation*}
 \begin{split}
  A^1 &= \{(i_1,j_1,i_2,j_2):0\leq i_1 < j_1 \leq i_2 < j_2\leq n \},\\
  A^2 &= \{(i_1,j_1,i_2,j_2):0\leq i_1 \leq i_2 < j_1 < j_2 \leq n \},\\
    A^3 &= \{(i_1,j_1,i_2,j_2):0\leq i_1 \leq i_2 < j_2 \leq j_1 \leq n \},\\
  B^1 &= \{(i_1,j_1,i_2,j_2):0\leq i_2 < j_2 \leq i_1 < j_1\leq n \},\\
  B^2 &= \{(i_1,j_1,i_2,j_2):0\leq i_2 < i_1 < j_1 \leq j_2 \leq n \},\\
  B^3 &= \{(i_1,j_1,i_2,j_2):0\leq i_2 < i_1 < j_2 < j_1 \leq n \}.
 \end{split}
\end{equation*}
The sums over the sets $A^1$, $B^1$ are zero by independence, and thus
\begin{equation}
\var(V_n) =4( a_2(n) + a_3(n) + b_2(n) + b_3(n)),
\end{equation}
where $a_2$, $a_3$, $b_2$, and $b_3$ are the sums over $A^2$, $A^3$, $B^2$ and $B^3$ respectively.

Other approachs such as local limit theorems(\cite{Law91}, Chap.6), or strong invariance principle(\cite{Bass06}) require finite moments of higher order.
On the other hand to apply Karamata's Tauberian theorem as in \cite{Bolt89, Cerny07}, the underlying sequence must be monotone. Bolthausen~\cite{Bolt89} circumvented this restriction of by considering the terms of the difference separately.
%
If we do consider them separately we get the exact asymptotic
$$c(n)= \sum\limits_{\mathbf{m}\in M_n}\rP(S_{m_2+m_3+m_4}=0)\rP(S_{m_3}=0)\thicksim Cn^2 \log(n),$$
since $\sum_{n} c(n) \lambda^n \sim C(1-\lambda)^{-3}\log(1/(1-\lambda))$, which is precisely the correct calculation for the terms $a_2$, $a_4$ in \cite{Bolt89}.
On the other hand \cite{Cerny07} considered the terms of the difference together. Letting $M_n$ be the set of 5-tuples $(m_1, \dots, m_5)$ such that $m_1,m_2,m_4,m_5\geq 0$, $m_3>0$ and $m_1 + \cdots +m_5 = n$,
and using  the formula $\rP(S_k = 0) = (2\pi)^{-d} \int_J f^k (x) \rd x$,
we can write
\begin{align*}
a_3(n)
&= (2\pi)^{-2}\sum_{\mathbf{m}\in M_n}
\rP(S_{m_3 = 0}) \int_{J}f^{m_2 + m_4}(k)\big[1 - f(k)^{m_3}\big] \rd k
\end{align*}
The monotonicity restriction of the Tauberian theorem in this case rougly requires that $f(t) \geq 0$.

\subsection{A Darboux-Wiener type lemma}
Rather than appealing to the Tau\-be\-rian theorem, the proof of Theorem~\ref{main} is based on Lemma~\ref{tauberian}.
Similar results have been used in the past in the context of singularity analysis, and in fact Lemma~\ref{tauberian} generalizes Theorem~4 in Flajolet and Odlyzko~\cite{Flaj90}, which mainly treats algebraic singularities. This approach which finds its origins in early work by Wiener \cite{Wiener32} and Darboux (see \cite{Knuth89} for Darboux's lemma), and has been well known in the combinatorial community for some time, is the key ingredient needed to revive the techniques in Bolthausen~\cite{Bolt89} and to correctly estimate the asymptotic variance of $V_n$.

\begin{lemma}\label{tauberian}
Assume that $g(z) = \sum_{n=0}^\infty a_n z^n$ is analytic for $|z|<1$.
Suppose there exists $\alpha \in (0,1)$, a constant $K>0$, such that
$|g(z)| \leq K$, for $\Re(z) \leq \alpha$, a sequence of
non-negative constants $A_m>0$, $\gamma_m >1$, and non-negative monotone increasing functions $l_m$
such that
$$|g(z)| \leq \sum_m A_m|1-z|^{-\gamma_m} l_m (|1-z|^{-1}),\quad \text{for $\Re(z) > \alpha$}.$$
Then
$$|a_n| \leq 4 K + \sum_m A_m C(\gamma_m) n^{\gamma_m -1}l_m(n),$$
where $C(\gamma)=4\pi^{-1/2}\Gamma(\frac{\gamma -1}{2})/\Gamma(\frac{\gamma}{2})$.
\end{lemma}

\begin{proof}
Let $\varGamma$ be a circle around the origin of radius $R=1-1/n$, for $n\geq 2$ and
$R=1/2$ for $n=1$.
We split $\varGamma$ in two arcs, $\varGamma_1 \equiv \{z\in \varGamma: \Re(z) \leq \alpha\}$, and $\varGamma_2 \equiv \{z\in \varGamma: \Re(z) > \alpha\}$.
By applying the Cauchy formula
\begin{align*}
|a_n|
&=\left|\frac{1}{2\pi i} \int_{\varGamma} g(z)z^{-n-1}\rd z\right|\\
&\leq \frac{1}{2\pi} \left| \int_{\varGamma_1} g(z) z^{-n-1}\rd z\right|
			+ \frac{1}{2\pi}\left|\int_{\varGamma_2} g(z)z^{-n-1}\rd z\right|.
\end{align*}
Since $|g(z)|\leq K$ when $\Re(z)\leq \alpha$, and $R^{-n}\leq 4$
for $n\geq 1$,
$$\left| \int_{\varGamma_1} g(z) z^{-n-1}\rd z\right|
\leq \int_{0}^{2\pi} K R^{-n} \rd t\leq 8\pi K\ . $$
On the other hand for the integral on $\varGamma_2$,
\begin{equation*}
 \left|\int_{\varGamma_2} g(z)z^{-n-1}\rd z\right|
\leq  \sum_m R^{-n} A_m\int_{-\pi/2}^{\pi/2} |1-\cis{R}{t}|^{-\gamma_m} l_m(|1-\cis{R}{t}|^{-1}) \rd t,
\end{equation*}
Fix $m$ in the sum. Let the summand be denoted by $I$  and to simplify
notation let us ignore the dependence on $m$. It remains to prove that
$$
I\leq 2\pi C(\gamma) A n^{\gamma -1}l(n).
$$
Then since \ $|1-\cis{R}{t}| = [(1-R)^2 + 2R(1-\cos (t))]^{1/2}$\
and $l$ is monotone increasing, observe that for all $t$ and $n$
\begin{equation*}
l(|1-\cis{R}{t}|^{-1})=
l\left( [n^{-2} + 2R(1-\cos(t))]^{-1/2}\right)
 \leq l(n)
\end{equation*}
which together with $R^{-n}\leq 4$ leads to the bound
$$ I \leq  4 l(n) A \int_{-\pi/2}^{\pi/2} |1-\cis{R}{t}|^{-\gamma} \rd t.$$
From $\cos(t) \leq 1-t^2/4$\ for $t\in [-\pi/2,\pi/2]$, it follows that
\begin{align*}
 \int_{-\pi/2}^{\pi/2} |1-\cis{R}{t}|^{-\gamma} \rd t
 &\leq \int_{-\pi/2}^{\pi/2} \left[(1-R)^2 + \frac{Rt^2}{2}\right]^{-\gamma/2} \rd t\\
 &\leq 4 n^{\gamma-1}\int_0^{\infty} \left[1 + t^2\right]^{-\gamma/2} \rd t
 	= 2\sqrt{\pi}\frac{\Gamma(\frac{\gamma -1}{2})}{\Gamma(\frac{\gamma}{2})} n^{\gamma-1},
\end{align*}
for all $\gamma >1$, and therefore
\begin{equation*}
I \leq 8\sqrt{\pi}
\frac{\Gamma(\frac{\gamma -1}{2})}{\Gamma(\frac{\gamma}{2})} A n^{\gamma -1}l(n)
 = 2\pi  C(\gamma) A n^{\gamma -1}l(n).\qedhere
\end{equation*}
\end{proof}

{\it A renewal type example\/}.
Let $T_1$, $T_2$, $\dots$ be a sequence of non-negative iid $\mathbb{Z}$-valued random variables, and define the strongly aperiodic random walk $Y_i = \sum_{j=1}^i T_i$, for $i\geq 1$.
For $n\in \mathbb{N}$ let $N_n := \sum_{j=1}^\infty \mathbf{1}_{\{Y_j \leq n\}}$ be the number of renewals taking place up to time $n$.
\begin{proposition}
Suppose $\mathrm{E}T_i = \mu \in (0,\infty)$, and $f(\lambda) = E[\lambda^{T_1}]=1 +(\lambda-1)\mu+
R(1-\lambda)$ where $|R(1-\lambda)|\leq |1-\lambda|^{\delta} l(|1-\lambda|^{-1})$ for some
$\delta>0$ and non-decreasing function $l$. Then
\begin{equation*}
|\rE N_n - n/\rE T| \leq Cn^{1-\delta} l(n) \;\; \mbox{and}\;\;
|\rE N_n^2 - n^2/(\rE T_i)^2| \leq Cn^{2-\delta}l(n).
\end{equation*}
In particular $N_n /n \to 1/\mu$ a.s. as $n\to \infty$.
\end{proposition}
{\it Proof\/}. For complex $\lambda$ let

$$a(\lambda) := \sum_{n=0}^\infty \lambda ^n \mathrm{E}(N_n), \quad  b(\lambda) := \sum_{n=0}^\infty \lambda ^n \mathrm{E}(N_n ^2), \quad |\lambda| <1.$$
An easy application of the renewal equation and Taylor analysis gives
\begin{eqnarray*}
a(\lambda) &=& \frac{f(\lambda)}{(1-\lambda)(1-f(\lambda))} =
\frac{1}{(1-\lambda)^2\mu} + O(|1-\lambda|^{-2+\delta} l(|1-\lambda|^{-1})\\
b(\lambda) &=& a(\lambda) +  2a(\lambda) \frac{f(\lambda)}{1-f(\lambda)}
=\frac{2}{(1-\lambda)^3\mu^2} + O(|1-\lambda|^{-3+\delta} l(|1-\lambda|^{-1}))
\end{eqnarray*}
Let $c_n = (n+1)/\mathrm{E}T$, $d_n = (n^2+3n+2)/(\mathrm{E}T)^2$ and define
$$A(\lambda) = \sum_{n=0}c_n \lambda ^n = \frac{1}{(1-\lambda)^2 \mathrm{E}T},\quad
B(\lambda) = \sum_{n=0}^\infty d_n \lambda ^n = \frac{2}{(1-\lambda)^3 (\mathrm{E}T)^2}.$$
Applying Lemma~\ref{tauberian} with $g(\lambda)=a(\lambda) - A(\lambda)$ and $g(\lambda) = b(\lambda) - B(\lambda)$ we prove the inequalities in the lemma.\\
To complete the proof, let $\rho>1$, and define $n_m = [\rho^m]$. Chebyshev's inequality and the previous calculations, give for any $\epsilon >0$
$$\mathrm{P}\left(\left|\frac{N_{n_m}}{n_m} - \frac{1}{\rE T}\right|\geq \epsilon\right)
\leq \frac{\rE (N_{n_m} - n_m / \rE T)^2}{\epsilon n_m ^2}\leq C[\rho^m]^{-\delta},$$
which is summable over $m$. The Borel-Cantelli lemma implies  a.s. convergence along the sequence $n_m$. The result follows from monotonicity of $N_n$ and a trick going back to Breiman, which can be found in \cite[Prop.2.2]{Plant02}

\subsection{Proof of Theorem~\ref{main}(i)}
Continuing from the decomposition given in the beginning of the section, let us first estimate $a_3(n)$.
\begin{equation}\label{a3}
\begin{split}
  a_3 (n)
&= \sum\limits_{A^3}\left[\rP(S_{i_1}
= S_{j_1},S_{i_2}=S_{j_2}) - 								
				 \rP(S_{i_1}=S_{j_1})\rP(S_{i_2}=S_{j_2})\right]\\
&= \sum\limits_{\mathbf{m}\in M_n}\rP(S_{m_3}=0)
	\bigl[ \rP(S_{m_2+m_4} = 0)
  	 - \rP(S_{m_2+m_3+m_4}=0)\bigr],
\end{split}
\end{equation}
where $M_n$ is the set of 5-tuples $(m_1,\dots, m_5)$ such that $m_1,m_2,m_4,m_5 \geq 0$, $m_3 >0$ and $m_1 + \cdots + m_5=n$, and using the characteristic function representation
\begin{align*}
  \rho_3 (\lambda)
  	&=  \sum_{n\geq 0} a_3 (n) \lambda^n\\
  	&= (1-\lambda)^{-2}(2\pi)^{-2} \iint\limits_{J^2} \frac{\lambda f(y)(1 - f(x))\rd x \rd y}
        {(1-\lambda f(x))^2(1-\lambda f(y))(1-\lambda f(x) f(y))},
\end{align*}
Similar power series $\rho_2 (\lambda)$ will be constructed and treated for the sequence $a_2 (n)$.
The complete computations are quite lengthy and involve asymptotic analysis of many multivariate integrals with complex parameter. However most of the integrals involved are analyzed in a similar fashion.
We show the key steps of the analysis for $\rho_3 (\lambda)$ and indicate important changes for $\rho_2 (\lambda)$.
\vskip 3pt
\subsubsection*{Lower bounds for $|1-\lambda f(t)|$ and $|1-\lambda f(t) f(s)|$}
To treat this integral we need lower bounds for quantities of the form $1-\lambda f(t)$.
Recall that $f(t) = 1 - \gamma |t| + R(t),$ where $R(t) = o(|t|)$, as $t\to 0$.
Let $\epsilon > 0$ be fixed and as close to zero as we desire.

First observe that outside the region $U_\epsilon = \{(t,s)\in J^2: |x|<\epsilon, |y|<\epsilon\}$
we have by aperiodicity of the random walk that there is a constant $C(\epsilon)$ such that $|f(t)| \leq C(\epsilon) <1$. This implies that there exists another constant $C>0$ such that
\begin{equation}\label{bbound}
|1-\lambda f(t)| \geq  C >0, \mbox{ and } |1-\lambda f(t) f(s)| \geq C >0.
\end{equation}

Since $R(t)=o(|t|)$ we have
that on the region $U_\epsilon$, $|R(t)| < \theta_\epsilon |t|$, for some positive $\theta_\epsilon \to 0$ as $\epsilon \to 0$.
Using the triangle inequality we have
\begin{align}\label{cbound1}
|1-\lambda f(t)| 	
&=	|1-\lambda + \lambda \gamma |t| - \lambda R(t)|\geq \left| |1-\lambda + \lambda\gamma |t|| -  |\lambda||R(t)|\right| \notag\\
&\geq  \left||1-\lambda + \lambda\gamma |t|| -  \theta_\epsilon |t|\right| \equiv h_\epsilon(t,\lambda)
\end{align}
and similarly for $|t|,|s|<\epsilon$ we have
\begin{equation}\label{cbound2}
|1-\lambda f(t)f(s)|
\geq 	|1-\lambda + \lambda \gamma (|t| + |s|)| - \Delta_\epsilon (|t|+ |s|) \equiv k_\epsilon(t,s,\lambda)
\end{equation}
where $\Delta_\epsilon  = \gamma^2 \epsilon + \gamma \theta_\epsilon + \theta_\epsilon^2$.
If for some $\alpha \in (0,1)$,
$\Re(\lambda) \leq \alpha$ then for $|x|<\epsilon$, using the real part as a lower bound we have
\begin{equation}\label{cbound5}
|1-\lambda f(t)| \geq 1-\alpha - (\gamma + \theta_\epsilon)\epsilon \geq C>0,
\end{equation}
for $\epsilon$ small enough.

Let $z_1 \equiv (1-\lambda)/|1-\lambda|$ and $z_2 \equiv \lambda \gamma$ and suppose now that
$\Re(\lambda) > \alpha$. Then
\begin{equation}\label{cbound3}
|z_1 + z_2 |t|| - \theta_\epsilon |t| \geq \Re(z_1 + z_2 |t|) - \theta_\epsilon |t| \geq C|t|,
\end{equation}
for $\epsilon$ small enough. If $|t|< \delta$, using the triangle inequality
\begin{equation}\label{cbound4}
|z_1 + z_2 |t|| - \theta_\epsilon |t| \geq 1 - |z_2|\delta - \theta_\epsilon \delta \geq C>0,
\end{equation}
for $\delta$ small enough.
\subsubsection*{Integral away from zero}
Let us first consider the integral outside the region $U_\epsilon$.
Thus we have using \eqref{bbound}-\eqref{cbound4} consecutively
\begin{eqnarray*}
\lefteqn{\Biggl|\,\, \iint\limits_{J^2 \cap \{|y|\geq \epsilon\}} \frac{\lambda f(y)(1 - f(x))\rd x \rd y}
        {(1-\lambda f(x))^2(1-\lambda f(y))(1-\lambda f(x) f(y))} \Biggr|}\\
&&\leq C \int\limits_{|x|<\epsilon} \frac{|x|\rd x }{|1-\lambda f(x)|^2}
 	\leq C\int_0^\epsilon \frac{x\rd x }{||1-\lambda+\lambda\gamma x| - \theta_\epsilon x|^2 } \\
&&\leq C \int_0^{\epsilon / |1-\lambda|} \frac{x\rd x }{||z_1+z_2 x| - \theta_\epsilon x|^2}
\leq C + C \int_\delta^{\epsilon/ |1-\lambda|} x^{-1}\rd x \\
&&\leq C\left(1+\log_+(|1-\lambda|^{-1})\right),
\end{eqnarray*}
The other cases follow similarly giving the same order and thus we have
\begin{equation*}
\rho_3 (\lambda) = (1-\lambda)^{-2}(2\pi)^{-2} \iint\limits_{U_\epsilon} \frac{\lambda f(y)(1 - f(x))\rd x \rd y}
        {(1-\lambda f(x))^2(1-\lambda f(y))(1-\lambda f(x) f(y))} + I(\lambda).
\end{equation*}
where $I(\lambda)$ is the error from integrating only over the region $U_\epsilon$ and satisfies
$|I(\lambda)| \leq C|1-\lambda|^{-2}\log_+(|1-\lambda|^{-1})$.  By \eqref{cbound5} for $\Re(\lambda) \leq \alpha$ and some constant $K$, we have $|\rho_3(\lambda)|\leq K$.
From here on we shall assume that $\Re(\lambda) > \alpha$.
\vskip 3pt
\subsubsection*{A generic integral}
Since we are integrating over $U_\epsilon$ we would like to use the expansion (i) under the integral sign to simplify the
calculations. This will introduce a new error $E$ in our expansion of $\rho_3 (\lambda)$ which is given by
\begin{multline*}
E = (1-\lambda)^{-2}(2\pi)^{-2} \iint\limits_{U_\epsilon} \frac{\lambda f(y)(1 - f(x))\rd x \rd y}
        {(1-\lambda f(x))^2(1-\lambda f(y))(1-\lambda f(x) f(y))}\\
     - (1-\lambda)^{-2}(2\pi)^{-2} \iint\limits_{U_\epsilon} \frac{\lambda \gamma |x|\rd x \rd y}
        {(1-\lambda +\lambda\gamma |x|)^2 (1-\lambda +\lambda \gamma |y|)(1-\lambda + \lambda\gamma |x| |y|)}.
\end{multline*}
To obtain a bound on this error and to simplify the calculations we successively examine the errors from replacing each term in the integrand by its expansion.

We demonstrate the calculation for the first case, since the rest of the errors can be modified and reduced to this.
In the following, $C(\epsilon)$ is a positive constant
depending on $\epsilon$ such that $C(\epsilon)\to 0$ as $\epsilon \to 0$. Using the expansion (i), \eqref{cbound1} and \eqref{cbound2} we have
\begin{align*}
|E_1|
&\leq |1-\lambda|^{-2} \iint\limits_{U_\epsilon} \frac{\bigl| f(y)(1 - f(x)) - \gamma |x|\bigr|\rd x \rd y}
        {|1-\lambda f(x)|^2 |1-\lambda f(y)| |1-\lambda f(x) f(y)|}\\
&\leq C(\epsilon)|1-\lambda|^{-2} \iint\limits_{U_\epsilon} \frac{|x|\rd x \rd y}
        {h_\epsilon (x,\lambda)^2 h_\epsilon (y,\lambda) k_\epsilon (x,y,\lambda)        }\\
&\leq C(\epsilon)|1-\lambda|^{-3} \int_0^{\infty}\!\!\!\int_0^{\infty}\!\!\!
 \frac{x \rd x \rd y}{\tilde{h}_\epsilon (x,\lambda)^2 \tilde{h}_\epsilon (y,\lambda) \tilde{k}_\epsilon (x,y,\lambda)}
\end{align*}
where it is convenient to write
\begin{align*}
\tilde{h}_\epsilon (x,\lambda) &= \left| z_1 + z_2 |x| \right| - \theta_\epsilon |x|,\\
\tilde{k}_\epsilon (x,y,\lambda) &= \left| z_1 + z_2 (|x|+|y|)\right| - \Delta_\epsilon (|x|+|y|).
\end{align*}
Using \eqref{cbound3} and \eqref{cbound4}
\begin{align*}
|E_1|
&\leq C(\epsilon)|1-\lambda|^{-3} \int_0^{\infty}\!\!\!\int_0^{\infty}\!\!\!
 \frac{ \rd x \rd y}{\tilde{h}_\epsilon (x,\lambda) \tilde{h}_\epsilon (y,\lambda) \tilde{k}_\epsilon (x,y,\lambda)}\\
&\leq C(\epsilon)|1-\lambda|^{-3}\left[C + \int_\delta^\infty\!\!\!\int_\delta^\infty \frac{\rd x \rd y}{xy(x+y)}\right]
\leq C(\epsilon)|1-\lambda|^{-3},
\end{align*}
uniformly in $\lambda$. The other errors follow similarly with the same bounds, giving the expansion
\begin{align*}
\rho_3(\lambda) &= 4(1-\lambda)^{-2}(2\pi)^{-2}\\
&\quad\times \int_0^\epsilon\!\!\int_0^\epsilon \frac{\lambda \gamma x\rd x \rd y}
        {(1-\lambda +\lambda\gamma x)^2 (1-\lambda +\lambda \gamma y)(1-\lambda + \lambda\gamma (x+ y))} +E + I
\end{align*}
where $|E| \leq C(\epsilon)|1-\lambda|^{-3}$ and $|I|\leq |1-\lambda|^{-2}\log_+|1-\lambda|^{-1}$.
\vskip 3pt
\subsubsection*{Moving from $U_\epsilon$ to $\mathbb{R}^2$}Finally
we simplify the integral by integrating over the positive half-axis, rather than $(0,\epsilon)$ giving
\begin{align*}
\rho_3(\lambda) &= 4(1-\lambda)^{-2}(2\pi)^{-2}\\
&\quad\times \int_0^\infty\!\!\!\int_0^\infty \frac{\lambda \gamma x\rd x \rd y}
        {(1-\lambda +\lambda\gamma x)^2 (1-\lambda +\lambda \gamma y)(1-\lambda + \lambda\gamma (x+ y))}\\
        &\qquad +E + I - F
\end{align*}
where $F$ is the integral over $V=[0,\infty)^2/[0,\epsilon)^2$. Observing that this region can be split
into $(\epsilon,\infty)\times [0,\epsilon)$, $[0,\epsilon)\times(\epsilon,\infty)$ and $(\epsilon,\infty)\times (\epsilon,\infty)$ we proceed to bound the integral $F$.
The first case is
\begin{align*}
|F_1|
&\leq C|1-\lambda|^{-2} \int_\epsilon^\infty\!\!\!\!\int_0^\epsilon \frac{x\rd x \rd y}
		{|1-\lambda+\lambda\gamma x|^2 |1-\lambda + \lambda \gamma y| |1-\lambda + \lambda\gamma (x+y)|}\\
&\leq C|1-\lambda|^{-3} \int_{\epsilon/|1-\lambda|}^\infty \int_0^{\epsilon/|1-\lambda|} \!\!\!\frac{x\rd x \rd y}
 		{\tilde{h}_\epsilon (x,\lambda)^2 \tilde{h}_\epsilon (y,\lambda) \tilde{k}_\epsilon (x,y,\lambda)}\\		
&\leq C|1-\lambda|^{-2}\log_+|1-\lambda|^{-1}.		
\end{align*}
The second case is identical by symmetry. Finally for the third case
\begin{align*}
|F_3|
&\leq C|1-\lambda|^{-3} \int_{\epsilon/|1-\lambda|}^\infty \int_{\epsilon/|1-\lambda|}^\infty \frac{\rd x \rd y}
		{xy(x+y)} \leq C|1-\lambda|^{-2}.		
\end{align*}
Assume for the moment that $\lambda$ is real and lies in the interval $(1/2,1)$ in order to calculate the integral
\begin{equation*}
\int_0^\infty\!\!\!\int_0^\infty\!\!\! \frac{\lambda \gamma x\rd x \rd y}
        {(1-\lambda +\lambda\gamma x)^2 (1-\lambda +\lambda \gamma y)(1-\lambda + \lambda\gamma (x+ y))}
= (1-\lambda)^{-1}(\lambda\gamma)^{-2}.
\end{equation*}
By analytic continuation this also holds for all complex $\lambda$ with $|\lambda|<1$.
Finally we now have
\begin{equation*}
\rho_3(\lambda) = (1-\lambda)^{-3}(\pi\gamma)^{-2} + \mathcal{E}
\end{equation*}
where $\mathcal{E}$ is the total error and as we have shown for $\Re(\lambda)>\alpha$ it satisfies
$$\mathcal{E} \leq C|1-\lambda|^{-2} + C|1-\lambda|^{-2}\log_+|1-\lambda|^{-1} + C(\epsilon)|1-\lambda|^{-3}.$$
If we let
$$g(\lambda) = \sum_{n=0}^\infty c_n \lambda^n = (\pi\gamma)^{-2}(1-\lambda)^{-3},$$
 by standard calculations we have $c_n = (n^2+3n+2)/2\pi^2\gamma^2$. By application of Lemma~\ref{tauberian} with $f(\lambda) = \rho_3(\lambda)-g(\lambda)$,
$$\left|a_3(n) - \frac{1}{2\pi^2\gamma^2}n^2\right| \leq D(\epsilon)n + D(\epsilon)n\log(n) + C(\epsilon)n^2,$$
where as $\epsilon \to 0$, $C(\epsilon)\to 0$ while $D(\epsilon)$ may be unbounded,
implying that $a_3(n)\thicksim n^2/2\pi^2\gamma^{2}$.

Let us now briefly consider the term $a_2(n)$.
\begin{align*}
 a_2 (n)&= \sum\limits_{A^2}
  	\left[\rP(S_{i_1} = S_{j_1},S_{i_2}=S_{j_2}) - \rP(S_{i_1}=S_{j_1})\rP(S_{i_2}=S_{j_2})\right]\\
  	&= \sum\limits_{\mathbf{m}\in M_n}\Bigl[\sum\limits_{x\in \mathbb{Z}} \rP(S_{m_2} = x)\rP(S_{m_3} = -x)\rP(S_{m_4} = x)\\
  	&\qquad - \rP(S_{m_2+m_3}=0)\rP(S_{m_3+m_4}=0)\Big],
\end{align*}
where $M_n$ is the set of 5-tuples $(m_1,\dots, m_5)$ such that $m_1,m_2,m_5 \geq 0$, $m_3, m_4 >0$ and $m_1 + \cdots + m_5=n$.
Then we have
\begin{align*}
   \rho_2 (\lambda)
  &=	(1-\lambda)^{-2}\lambda^2 (2\pi)^{-2}\\
  &\quad \times
  	\iint_{J^2}  \frac{ f(x)\rd x \rd y}{(1-\lambda  f(x))(1-\lambda  f(y))}
  	\left[
  		\frac{ f(x +y)}{1-\lambda f(x+y)} - \frac{ f(y)^2}{1-\lambda  f(x) f(y)} \right]
\end{align*}
and by calculations similar to the first term we have
\begin{multline*}
\rho_2 (\lambda)
= (1-\lambda)^{-2}\lambda^2 (2\pi)^{-2}\\
\times
 \biggl[\iint_{\mathbb{R}^2}
	\frac{\rd x \rd y}{(1-\lambda +\lambda\gamma |x|) (1-\lambda +\lambda\gamma |y|) (1-\lambda +\lambda\gamma |x+y|)}\\
 - \iint_{\mathbb{R}^2}
\frac{\rd x \rd y}{(1-\lambda +\lambda\gamma |x|) (1-\lambda +\lambda\gamma |y|) (1-\lambda +\lambda\gamma (|x|+|y|))} \biggr]  + \mathcal{E},
\end{multline*}
where $\mathcal{E}$ is the total error and satisfies
$$|\mathcal{E}| \leq C|1-\lambda|^{-2} + C|1-\lambda|^{-2}\log_+|1-\lambda|^{-1} + C(\epsilon)|1-\lambda|^{-3}.$$
As before, for $\lambda \in (1/2,1)$,
\begin{gather*}
\iint_{\mathbb{R}^2}
	\frac{\rd x \rd y}{(1-\lambda +\lambda\gamma |x|) (1-\lambda +\lambda\gamma |y|) (1-\lambda +\lambda\gamma |x+y|)}
=	(1-\lambda)^{-1} \left(\frac{\pi}{\lambda\gamma}\right)^2,\\
\iint_{\mathbb{R}^2}
\frac{\rd x \rd y}{(1-\lambda +\lambda\gamma |x|) (1-\lambda +\lambda\gamma |y|) (1-\lambda +\lambda\gamma (|x|+|y|))}
= \frac{2}{3}(1-\lambda)^{-1}\left(\frac{\pi}{\lambda\gamma}\right)^2.
\end{gather*}
By analytic continuation these two expressions hold for complex $\lambda$ with $|\lambda|<1$ and
thus we have the following expansion for $\rho_2 (\lambda)$,
\begin{equation*}
\rho_2(\lambda) = \frac{1}{12}\gamma^{-2} (1-\lambda)^{-3} + \mathcal{E}.
\end{equation*}
Thus using Lemma~\ref{tauberian} and calculations similar to $\rho_3(\lambda)$, we have that
$a_2 (n)\thicksim n^2/24\gamma^2$.

It is straightforward to show that $b_2(n)\thicksim a_3(n)$ and $b_3(n)\thicksim a_2 (n)$ and thus
we have that
$$\var(V_n) \thicksim 4\left( \frac{1}{12\gamma^2} + \frac{1}{\pi^{2}\gamma^{2}} \right)n^2.$$
\subsection{Proof of Theorem~\ref{main}(ii)}
We now consider the case with $d=2$, with a non-singular covariance matrix $\Sigma$ which implies that
$f(t) = 1 - \frac{1}{2}\braket{\Sigma t}{t} + R(t),$ where $R(t) = o(|t|^2)$, as $|t| \to 0$
for $t\in J=[-\pi,\pi)^2$.

By working with complex $\lambda$ and applying Lemma~\ref{tauberian}, we are able to avoid the extra assumptions on the random walk present required in \cite{Bolt89,Cerny07} as discussed in the beginning of this section.

We continue with our calculation, with $a_3(n)$ as defined in \eqref{a3}.
\begin{equation*}
  \rho_3 (\lambda)
  	= (1-\lambda)^{-2}(2\pi)^{-4} \iint_{J^2} \frac{\lambda f(t_2)(1 -  f(t_1))\rd t_1 \rd t_2}
        {(1-\lambda f(t_1))^2(1-\lambda f(t_2))(1-\lambda f(t_1) f(t_2))},
\end{equation*}
where $J = [-\pi,\pi)^2$.
Using the Taylor expansion of $f$ we are able to deduce lower bounds for the quantities
$|1-\lambda f(t_1)|$ and $|1-\lambda f(t_1) f(t_2)|$ for $|t_1|,|t_2|<\epsilon$. For convenience we
write $g(t_1,t_2)\equiv \braket{\Sigma t_1}{t_1}+\braket{\Sigma t_2}{t_2}$.
\begin{align*}
|1-\lambda f(t_1)| &\geq \left| \left|1-\lambda + \tfrac{\lambda}{2}\braket{\Sigma t_1}{t_1}\right| - \theta_\epsilon \braket{\Sigma t_1}{t_1}\right|\\
|1-\lambda f(t_1) f(t_2)| &\geq \left| \left|1-\lambda + \tfrac{\lambda}{2}g(t_1,t_2)\right| - \Delta_\epsilon g(t_1,t_2)\right|
\end{align*}
and for $z_1\equiv (1-\lambda)/|1-\lambda|$, $z_2 = \lambda/2$ we have for $\Re(\lambda)>\alpha \in (0,1)$
\begin{align*}
|z_1 + z_2 \braket{\Sigma t_1}{t_1}| - \theta_\epsilon \braket{\Sigma t_1}{t_1} &\geq C(1\wedge \braket{\Sigma t_1}{t_1})\\
|z_1 + z_2 g(t_1,t_2)| - \Delta_\epsilon g(t_1,t_2)
&\geq C (1 \wedge g(t_1,t_2))
\end{align*}
for positive $\theta_\epsilon, \Delta_\epsilon \to 0$ as $\epsilon\to 0$.
Using these bounds we can estimate
that the integral $I$ outside $U_\epsilon$ satisfies
$$|I| \leq C(\epsilon)|1-\lambda|^{-3} + C|1-\lambda|^{-2 } \log_+|1-\lambda|^{-1},$$
where $C(\epsilon)>0$ is a constant such that $C(\epsilon)\to 0$ as $\epsilon \to 0$. Once again for $\Re(\lambda)\leq \alpha$ we have
$|\rho_3(\lambda)|\leq K$. Hence we assume from now on that $\Re(\lambda)> \alpha$.
Thus we have
\begin{align*}
\rho_3 (\lambda)
&=(1-\lambda)^{-2}(2\pi)^{-4}|\Sigma|^{-1} \\
     &\qquad\times\iint\limits_{U_\epsilon} \frac{\frac{\lambda}{2}|t_1|^2\rd t_1 \rd t_2}
        {(1-\lambda + \tfrac{\lambda}{2}|t_1|^2)^2(1-\lambda+\tfrac{\lambda}{2}|t_2|^2) (1-\lambda + \frac{\lambda}{2}(|t_1|^2 +|t_2|^2))}\\
       &\qquad + I + E
\end{align*}
where $E$ is the error arising from using again the Taylor expansion under the integral sign.
In a similar manner to Section 2 we obtain a bound for the error $E$
$$|E| \leq C(\epsilon)|1-\lambda|^{-3}.$$
Finally we can replace the area of integration by the real plane giving error
$$|F|\leq C|1-\lambda|^{-2} \log_+|1-\lambda|^{-1}.$$
Now for real $\lambda\in (1/2,1)$ we have after changing to polar coordinates
\begin{align*}
\lefteqn{ \iint\limits_{\mathbb{R}^2\times \mathbb{R}^2} \frac{\frac{\lambda}{2} |t_1|^2\rd t_1 \rd t_2}
        {(1-\lambda + \tfrac{\lambda}{2}|t_1|^2)^2(1-\lambda+\tfrac{\lambda}{2}|t_2|^2) (1-\lambda + \frac{\lambda}{2}(|t_1|^2 +|t_2|^2))}}\\
&=(2\pi)^2 (1-\lambda)^{-1}\lambda^{-2}\int_0^\infty\!\!\!\int_0^\infty \frac{r\rd r \rd s}
        {(1 + r)^2(1+s) (1 + r +s)}        \\
&=(2\pi)^2 \lambda^{-2} (1-\lambda)^{-1}.
\end{align*}
Thus by analytic continuation we have for all $|\lambda|<1$ that
$$\rho_3(\lambda) = (2\pi)^{-2}|\Sigma|^{-1}(1-\lambda)^{-3} + \mathcal{E},$$
where $\mathcal{E}$ is the total error and satisfies
$$|\mathcal{E}| \leq C(\epsilon)|1-\lambda|^{-3} + C|1-\lambda|^{-2} + C|1-\lambda|^{-2}\log_+|1-\lambda|^{-1}.$$
Application of Lemma~\ref{tauberian} with $g(\lambda) = \rho_3(\lambda) - (2\pi)^{-2}|\Sigma|^{-1}(1-\lambda)^{-3},$ \ gives
$$\Big|a_3(n) - \frac{1}{8\pi^2|\Sigma|}n^2 \Big|\leq C(\epsilon)n^2+D(\epsilon)n\log (n) + D(\epsilon)n\ ,$$
where again, as $\epsilon \to 0$, $C(\epsilon)\to 0$ while $D(\epsilon)$ may be unbounded,
implying  $a_3(n) \thicksim n^2/8\pi^2 |\Sigma|$.

Calculation of the order of $a_2(n)$ is similar giving
$$a_2(n)\thicksim \frac{1}{2}(2\pi)^{-2}|\Sigma|^{-1} \kappa n^2,$$
where $\kappa$ was defined in Theorem (ii). Finally it is straightforward to show that
$b_2(n)\thicksim a_3(n)$ and $b_3(n)\thicksim a_2 (n)$ implying the desired approximation
$$\var(V_n) \thicksim 4(2\pi)^{-2}|\Sigma|^{-1}(1+\kappa)n^2.$$
\subsection{Proof of Theorem~\ref{thmbounds}}
We only consider a two-dimensional aperiodic random walk in $\mathbb{Z}^2$
with characteristic function satisfying
$|1-f(t)|\leq \alpha |t|^2$, for $|t|\in [-\pi,\pi)^2$ and
$|f(t)|\leq \re^{-\alpha |t|^2}$ for some $\alpha >0$.
Under these assumptions we also have $|1-f^j(t)| \leq C ((t^2 j)\wedge 1)$.

Then we have for $a_3(n)$ defined in \eqref{a3}
\begin{align*}
a_3 (n)
&= \sum\limits_{\mathbf{m}\in M_n}
\left[ \rP(S_{m_2+m_4} = 0,S_{m_3}=0) - \rP(S_{m_2+m_3+m_4}=0)\rP(S_{m_3}=0)\right]\\
&= (2\pi)^{-4}\sum\limits_{\mathbf{m}\in M_n} \iint_{J^2}f^{m_3}(t) f^{m_2 + m_4}(y) (1-f^{m_3}(y)) \rd t \rd y
\end{align*}
where $M_n$ is the set of 5-tuples $(m_1,\dots, m_5)$ such that $m_1,m_2,m_4,m_5 \geq 0$ and $m_3 >0$.
Then for $J=[-\pi,\pi)^2$
\begin{align*}
|a_3 (n)|
&\leq C\sum\limits_{\mathbf{m}\in M_n} \iint\limits_{J^2} |f(t)|^{m_3} |f(y)|^{m_2+m_4} |1-f^{m_3}(y)|\rd t \rd y\\
&\leq Cn \sum_{i=1}^n\sum_{j=0}^n (j+1)\Big(\int_0^\pi \re^{-i\alpha r} \rd r\Big)
	\Big(\int_0^\pi \re^{-j\alpha x}((ix)\wedge 1)\rd x\Big)\\
&\leq Cn \sum_{i=1}^n\sum_{j=0}^n (j+1)\frac{1}{i\alpha}
\Big[ \mathbf{1}_{\{j=0\}} + \Big(\frac{i}{\alpha^2 j^2}\Big(1\wedge \frac{\alpha j}{i} \Big) + \frac{\re^{-\alpha j/i}}{j}\Big)\mathbf{1}_{\{j\geq 1\}}\Big]\\
&=cn \sum_{i=1}^n \frac{1}{\alpha i}\Big[1 + \alpha^{-2} (1+\alpha)\sum_{j=1}^n\frac{2i}{j}\Big(1\wedge\frac{j}{i}\Big) + 2\sum_{i=1}^n \re^{-\alpha j/i}\Big]\\
&\leq C(\alpha^{-1} + \alpha^{-3})n \sum_{i=1}^n \left(\ln(\frac{n}{i}) + 1\right)\\
&\leq C(\alpha^{-1} + \alpha^{-3}) n^2 \int_0^1\Big(  \ln \Big(\frac{1}{x}\Big) + 1\Big)\rd x \leq C(\alpha^{-1} + \alpha^{-3})n^2.
\end{align*}

To treat $a_2(n)$ we apply Fourier transform to the terms in the difference separately (c.f. \cite{Bolt89}).
\begin{multline*}
\sum\limits_{\mathbf{m} \in M_n} \rP(S_{m_2 + m_3})\rP(S_{m_3 + m_4})
\leq \sum\limits_{\mathbf{m} \in M_n} \iint_{\mathbb{R}^4} \re^{-\alpha |t|^2 (m_2 + m_3) } \re^{-\alpha |y|^2 (m_3 + m_4)} \rd t \rd y\\
\leq \frac{C n}{\alpha^2} \sum\limits_{m_2,m_3,m_4} \frac{1}{m_2 + m_3} \frac{1}{m_3 + m_4}
\leq \frac{Cn}{\alpha^2} \sum\limits_{m_3 =1}^{n} \left(\sum\limits_{i=m_3}^{2n} \frac{1}{i}\right)^2 \leq \frac{Cn^2}{\alpha^2}
\end{multline*}

To treat the other term of the difference, we first condition on $S_{m_3}$ and then apply Fourier transform to get
\begin{align*}
&\sum\limits_{\mathbf{m} \in M_n} \rP(S_{m_2} = -S_{m_3}, S_{m_4} = -S_{m_3})\\
&\leq
\sum\limits_{\mathbf{m} \in M_n} \iint_{\mathbb{R}^4} \re^{-\alpha m_2 |t|^2 }\re^{-\alpha m_4 |y|^2 }
\re^{-\alpha m_3 |t+y|^2 } \rd t \rd y\\
&\leq \frac{C n}{\alpha ^2} \sum\limits_{m_2,m_3,m_4}
\!\!\!\big[ (m_2 + m_3)(m_4 + m_3)(m_2 m_4 + m_3 (m_2 + m_4))\big]^{-1/2} \leq \frac{C n^2}{\alpha ^2},
\end{align*}
where to compute the integral we diagonalized the quadratic form.

\section{Proof of Theorem~\ref{clt}}
We prove weak convergence of the laws of $Y_n(t)$ in $D[0,1]$ by first showing convergence of the finite dimensional distributions and then proving tightness. Before we can proceed with the proof we need a few results on the local times and self-intersections of the random walk.
\subsection{Auxiliary results\label{aux}}
Letting $N_n (x) = \sum_{i=0}^n \mathbf{1}_{\{S_i = x\}}$ denote the local time at site $x$ up to time $n$ we can write
$$Z_n = \sum_{i=0}^n \xi (S_i) = \sum_{x\in \mathbb{Z}} N_n (x) \xi _x.$$
In this direction we now give the following result on the asymptotic moments of the self-intersection local time $V_n$,
For $\alpha \in \mathbb{Z}$, let $N_{\alpha}(n) = \sum_{j=0}^n \mathbf{1}_{S_j = \alpha}$.
\begin{lemma}\label{lem:aux}
Under the assumptions of Theorem~\ref{clt} then
$\rE(V_n) \thicksim 2n \log n/\pi \gamma$,
$V_n / \rE V_n \to 1$, and $\sup_{\alpha} N_{\alpha}(n) = o(n^\epsilon)$, a.s. for each $\epsilon >0$.
If in addition $0<a<b$ then
$$\sum_{j=1}^{[an]}\sum_{i=[an]+1}^{[bn]} \mathbf{1}(S_i = S_j) = o(n\log n) \text{a.s. as $n\to \infty$}.$$
\end{lemma}
The proof of a.s. convergence of $V_n / \rE V_n \to 1$ is essentially given in \cite{Plant08} but relies heavily on the bound $\var(V_n) = O(n^2)$.
The rest be easily adapted from \cite{Bolt89}.

We are now ready to show convergence of finite dimensional distributions.
Let $a_1,\dots, a_m \in \mathbb{R}$, $0=t_0<t_1<\cdots < t_m$ be given and write
\begin{equation*}\label{eqn1}
\sum_{j=1}^m a_j (Y_n (t_j) - Y_n (t_{j-1}))
= \sum_{\alpha \in \mathbb{Z}} \sum_{j=1}^m a_j \big(N_{\alpha}([nt_j]) - N_{\alpha}([nt_{j-1}]) \big)\xi(\alpha)/d_n.
\end{equation*}
where $d_n = \sigma \sqrt{2n \log n}/ \sqrt{\pi \gamma}$.
Let $\mathcal{A} = \sigma (X_1,X_2,\dots)$, the $\sigma$-algebra generated by the random walk increments. Conditional on $\mathcal{A}$, the above expression is a sum of independent random variables with non-identical distributions.
To simplify notation we write
\begin{align*}
s_n^2 &=d_n^{-2}\sigma^2 \sum_{\alpha \in \mathbb{Z}} \Big(\sum_{j=1}^m a_j \big(N_{\alpha}([nt_j]) - N_{\alpha}([nt_{j-1}]) \big)\Big)^2,\\
C(n,\alpha) &= d_n ^{-1}\sum_{j=1}^m a_j \big(N_{\alpha}([nt_j]) - N_{\alpha}([nt_{j-1}])\big).
\end{align*}
We proceed by checking if the Lindeberg condition is satisfied conditionally on $\mathcal{A}$.
It suffices to show that for all $\epsilon >0$, for a.e. path
$$s_n ^{-2} \sum_{\alpha \in \mathbb{Z}}\rE\Bigl[ C(n,\alpha)^2 \xi(\alpha)^2 \mathbf{1}\{C(n,\alpha)\xi(\alpha) \geq \epsilon s_n\} \Bigl| \mathcal{A} \Bigr] \to 0,\quad n\to \infty.$$
Using the results of Lemma~\ref{lem:aux} it can be shown that
$$d_n^{-2} \sum_{\alpha \in \mathbb{Z}} \Big(\sum_{j=1}^m a_j \big(N_{\alpha}([nt_j]) - N_{\alpha}([nt_{j-1}]) \big)\Big)^2
\to \sigma^{-2} \sum_{j=1}^m a_j ^2 (t_j - t_{j-1}),$$
a.s. as $n\to \infty$, while by Lemma~\ref{lem:aux}(ii) we have
$$\sum_{j=1}^m a_j \big(N_{\alpha}([nt_j]) - N_{\alpha}([nt_{j-1}]) \big) = o(n^\delta)\quad a.s.$$
as $n\to \infty$ for any $\delta >0$.
These facts together imply that $s_n/C(n,\alpha) \to \infty$,
and by the square integrability of $\xi(\alpha)$
%
%
%
\begin{align*}
\lefteqn{s_n^{-2}\sum_{\alpha \in \mathbb{Z}}\rE\Big( C(n,\alpha)^2 \xi(\alpha)^2 \mathbf{1}\{\xi(\alpha)^2 \geq \epsilon s_n^2/C(n,\alpha)^2 \}\Big| \mathcal{A} \Big.\Big)}\\
&= C \rE\Big(  \xi(\alpha)^2 \mathbf{1}\{\xi(\alpha)^2 \geq \epsilon s_n^2/C(n,\alpha)^2 \}\Big| \mathcal{A} \Big.\Big) \to 0,
\end{align*}
as $n\to \infty$. Thus the Lindeberg condition is satisfied conditionally on $\mathcal{A}$ for almost every path of the random walk and thus by the central limit theorem and the fact that a.s.
%
%
$$d_n^{-2} \sigma^2 \sum_{\alpha \in \mathbb{Z}} \Big(\sum_{j=1}^m a_j \big(N_{\alpha}([nt_j]) - N_{\alpha}([nt_{j-1}]) \big)\Big)^2
\to \sum_{j=1}^m a_j^2 (t_j - t_{j-1}),$$
we can conclude that
$$\sum_{j=1}^m a_j (Y_n(t_j) - Y_n(t_{j-1}))\cvgdist N(0,{\textstyle\sum_{j=1}^m a_j^2 (t_j - t_{j-1})}).$$
Convergence of the finite dimensional distributions follows from the Cram\'er-Wold theorem.

Tightness then follows by carefully adapting the Bolthausen proof. Alernatively, we truncate in terms of monotone functions $\xi_x = f_{M^+}(\xi_x) + f_{M^-}(\xi_x) + f^M (\xi_x)$, where $f^M(y) = y$ for $|y|\leq M$ and $M$ otherwise, $f_{M^+}(y)=y-M$ for $y>M$ and 0 otherwise, and $f_{M^-}(y) = y+M$ for $y< -M$ and 0 otherwise. By the Newman-Wright maximal inequality (see \cite{Newman81}) it follows that the left and right parts of the scenery, corresponding to $f_{M^+}$ and $f_{M^-}$ converge to zero. Tightness of the truncated scenery follows from the maximal inequality in \cite[Theorem~3.1]{Mor82}.

\end{document}